\documentclass[a4paper,12pt]{article}
\usepackage{amsmath,amssymb,amsthm,graphics,latexsym,amsfonts}
\usepackage{fancyhdr}
\usepackage{graphics}
\usepackage{color}

\title{\Large Construction and characterization  of  graphs whose each spanning tree has a perfect matching\thanks{This research was supported by
NSFC (Grant No. 11571294 and No. 11371180)}}
\author{ {Baoyindureng Wu \footnote{Email: baoywu@163.com (B. Wu)} }\\
\small  College of Mathematics and System Sciences, Xinjiang
University, \\ \small  Urumqi, Xinjiang 830046, P.R.China \\
{Heping Zhang \footnote{
Email: zhanghp@lzu.edu.cn  (H. Zhang) }}\\
\small  School of Mathematics and Statistics, Lanzhou University,\\
\small Lanzhou, Gansu 730000, PR China}
\date{}

\newtheorem{theorem}{Theorem}[section]
\newtheorem{lemma}[theorem]{Lemma}
\newtheorem{corollary}[theorem]{Corollary}

\usepackage{indentfirst}

\begin{document}
\maketitle

\begin{abstract} An edge subset $S$ of a connected graph $G$ is called an
anti-Kekul\'{e} set  if $G-S$ is connected and has no perfect
matching. We can see that a connected graph $G$ has no anti-Kekul\'{e} set if and only if each spanning tree of $G$
has a perfect matching. In this paper, by applying Tutte's 1-factor theorem and structure of minimally 2-connected graphs, we characterize all graphs whose each spanning tree has a perfect matching
In addition, we show that if $G$ is a connected graph of order $2n$ for a positive integer
$n\geq 4$ and size $m$ whose each spanning tree has a perfect
matching, then $m\leq \frac{(n+1)n} 2$, with equality if and only if $G\cong K_n\circ K_1$.\\

{\bfseries Keywords}: Perfect matching; Anti-Kekul\'{e} set;
Spanning tree; Minimally 2-connected graph

\end{abstract}

\section{Introduction}
All graphs considered in this paper are finite and simple. We refer to \cite{Bondy2}
for  undefined notation and
terminology. For a graph $G=(V(G), E(G))$,
we denote the {\em order} and the {\em size} of $G$, respectively, by $v(G)$ and $e(G)$.
For a vertex $v\in V(G)$, the {\em degree} of $v$, denoted by $d_G(v)$,
is the number of edges incident with $v$ in $G$; the {\em neighborhood} of $v$, denoted by
$N_G(v)$, is the set of vertices adjacent to $v$ in $G$.
As usual, the complete graph, the path and the cycle of
order $n\geq 1$ are denoted by $K_n$, $P_n$ and $C_n$, respectively.
A {\it matching} in a graph $G$ is a set of pairwise nonadjacent
edges. If $M$ is a matching, the two ends of each edge of $M$ are
said to be {\em matched} under $M$, and each vertex incident with an
edge of $M$ is said to be {\em covered} by $M$. A {\em perfect
matching} (or {\em Kekul\'{e} structure} in chemistry) of a graph $G$ is a matching  which
covers every vertex of $G$. An edge of G is a {\em fixed
double (single)} edge if it belongs to all (none) of the perfect
matchings of $G$. Both fixed double edges and fixed single edges are
called {\em fixed} edges. A bipartite graph with a perfect matching
is called {\em normal (or elementary)} if it is connected and has no
fixed edges.

Let $G$ be a connected graph. An edge subset $S$ of $G$ is called an
{\em anti-Kekul\'{e} set} of $G$ if $G-S$ is connected and has no perfect
matching. The cardinality of a minimum anti-Kekul\'{e} set of $G$ is called
the anti-Kekul\'{e} number and is denoted by $ak(G)$. The
notion of anti-Kekul\'{e} set and anti-Kekul\'{e} number were first
introduced by Vuki\v{c}evi\'{c} and Trinajsti\'{c} \cite{Vuki} in
2007.

 Vuki\v{c}evi\'{c} and
Trinajsti\'{c} \cite{Vuki, Vukic} showed that the anti-Kekul\'{e}
number of benzenoid parallelograms is 2 and the anti-Kekul\'{e}
number of cata-condensed hexagonal systems equals either 2 or 3. Cai
and Zhang \cite{Cai} showed that for a hexagonal system $H$ with
more than one hexagon, $ak(H) = 0$ if and only if $H$ has no perfect
matching, $ak(H) = 1$ if and only if $H$ has a fixed double edge,
and $ak(H)$ is either 2 or 3 for the other cases. Further by
applying perfect path systems they gave a characterization whether
$ak(H) = 2$ or 3, and present an $O(n^2)$ algorithm for finding a
smallest anti-Kekul\'{e} set in a normal  hexagonal system, where
$n$ is the number of its vertices.

Vuki\v{c}evi\'{c} \cite{Vuk} showed that the anti-Kekul\'{e} number
of the icosahedron fullerene C$_{60}$ (buckminsterfullerene) is 4.
Kutnar et al. \cite{Kutnar}
proved that the anti-Kekul\'{e} number of all fullerenes is either 3
or 4 and that for each leapfrog fullerene the Anti-Kekul\'{e} number
can be established by observing finite number of cases not depending
on the size of the fullerene. Yang et
al. \cite{Yang} showed that the anti-Kekul\'{e} number is always
equal to 4 for all fullerene graphs.

Veljan and Vuki\v{c}evi\'{c} \cite{Veljan} found that the values of
the anti-Kekul\'{e} numbers of the infinite triangular, rectangular
and hexagonal grids are, respectively, 9, 6, 4. Among other things,
it was shown that the anti-Kekul\'{e} number of cata-condensed
phenylenes is 3 in \cite{Zhang}. Ye \cite{Ye} showed that, if $G$ is
a cyclically $(r+1)$-edge-connected $r$-regular graph $(r\geq 3)$ of
even order, then either the anti-Kekul\'{e} number of $G$ is at
least $r+1$, or $G$ is not bipartite, and the smallest odd cycle
transversal of $G$ has at most $r$ edges. L\"{u} et al. \cite{Lu} showed that computing the anti-kekul\'{e} number
of bipartite graphs is NP-complete.

In spite of the above known results on anti-Kekul\'{e} number of a
graph, a fundamental problem is not yet solved: which graphs do not
have an anti-Kekul\'{e} set~? Indeed, there exist some connected
graphs, for instance, $K_2$ and all even cycles, which do not have a
anti-Kekul\'{e} set. The aim of this note is to characterize all
these graphs. Our approach is to construct all such graphs from
$K_2$ and all even cycles. To this end, let us define recursively a
family $\mathcal{G}$ of graphs.

(1) $K_2$ and all even cycles belong to $\mathcal{G}$;

$(2)$ Assume that $H$ is a connected graph of order $p\geq 2$ with
vertex set $\{u_1, \ldots, u_p\}$ and $F_i\in \mathcal{G}$ for each
$i\in\{1, \ldots, p\}$, where $H, F_1, \ldots, F_p$ are pairwise
vertex-disjoint. For each $i$, take a vertex $v_i\in V(F_i)$. The
graph obtained from $H, F_1, \ldots, F_p$ by identifying the
vertices $u_i$ and $v_i$ for each $i$, denoted by $H[F_1(u_1v_1),
\ldots, F_p(u_pv_p)]$ (or simply by $H[F_1, \ldots, F_p]$), belongs
to $\mathcal{G}$.

The {\em corona} $G\circ K_1$ of a graph $G$ is
the graph obtained from $G$ by adding an edge between each vertex of $G$ and its  copy. Observe that $G\circ K_1\in \mathcal{G}$ for any connected graph $G$. The join of two vertex-disjoint graphs $G$ and $H$, denoted by $G\vee H$, is the graph obtained from $G\cup H$ by joining each vertex of $G$ to all vertices of $H$.

We present our main theorem as follows.

\begin{theorem} Let $G$ be a connected graph. The following
statements are equivalent:\\
(1) $G$ has no anti-Kekul\'{e} set,\\
(2) Each connected spanning subgraph of $G$ has a perfect matching,\\
(3) Each spanning tree of $G$ has a perfect matching,and\\
(4) $G\in \mathcal{G}$.
\end{theorem}

\section{Preliminary}
We start with Tutte's 1-factor theorem.
\begin{theorem} (Tutte \cite{Tutte})\label{tutter} A graph $G$ has a perfect matching if and only if $c_o(G-S) \leq |S|$
for any $S\subseteq V(G)$, where
$c_o(G-S)$ is the number of odd components of $G-S$.
\end{theorem}

For an integer $k\geq 1$, a $k$-connected graph $G$ is called {\em
minimally $k$-connected} if $G-e$ is not $k$-connected for each
$e\in E(G)$. The following property for a minimally 2-connected graph can be found
in \cite{Bollobas}.

\begin{theorem}(Bollob\'{a}s \cite{Bollobas}) Let $G$ be a minimally 2-connected graph that is not
a cycle. Let $V_2\subseteq V(G)$ be the set of vertices of degree
two. Then $F=G-V_2$ is a forest with at least two components. A
component $P$ of $G[V_2]$ is a path and the endvertices of $P$ are
not joined to the same tree of the forest $F$.
\end{theorem}

A {\em cut vertex} of a graph $G$ is a vertex $v$ such that
$c(G-v)>c(G)$, where $c(G)$ denotes the number of components of $G$.
A {\em decomposition} of a graph $G$ is a family $\mathcal{F}$ of
edge-disjoint subgraphs of $G$ such that
$\cup_{F\in\mathcal{F}}E(F)=E(G)$. A {\em separation} of a connected
graph is a decomposition of the graph into two nonempty connected
subgraphs of orders at least two which have just one vertex in common. This common vertex
is called a {\em separating vertex} of the graph. A cut vertex is
clearly a separating vertex. Since the graph under consideration is
simple, the two concepts,  separating vertex and cut vertex, are
identical. A graph is nonseparable if it is connected and has no
cut vertices; otherwise, it is separable. So a nonseparable graph other than $K_2$ is 2-connected.
 A {\em block} of a
graph $G$ is a subgraph which is nonseparable and is maximal with
respect to this property. Further, a block of $G$ is called an {\em end
block} if it contains just one cut vertex of $G$.

\vspace{3mm} To show our main theorem, we need the following two
lemmas.

\begin{lemma} If $G$ is a nonseparable graph whose
each spanning tree has a perfect matching, then it is isomorphic to
$K_2$ or an even cycle.
\end{lemma}
\begin{proof} By the assumption,  $G$ has even order $n$. Suppose to the contrary that
$G$ is isomorphic to neither $K_2$ nor an even cycle. Then $n\geq
3$ and $G$ is 2-connected.
We consider the following two cases.

\vspace{3mm} {\bf Case 1.} $G$ contains a Hamilton cycle $C$.

Label the vertices of $C$ as $v_1, v_2, \ldots, v_n$ in
the cyclic order. Since $G\not\cong C_n$, there is a chord for $C$.
Without loss of generality, let $v_1v_k$ be such an edge. Since
$3\leq k\leq n-1$, $T=C-v_{k-2}v_{k-1}-v_{k+1}v_{k+2}+v_1v_k$ is a spanning
tree of $G$ without a perfect matching, a contradiction.

\vspace{3mm}  {\bf Case 2.} $G$ contains no Hamilton cycle.

 Let $H$ be a minimally 2-connected spanning subgraph of
$G$. Then $H$ is not a cycle. Let $V_2$ be the set of vertices with degree 2 in $H$. By
Theorem 2.2, $H-V_2$ is a forest $F$ with at least two components.
Take a vertex $v\in V(F)$ with $0\leq d_F(v)\leq 1$. Since
$d_H(v)\geq 3$, $v$ has two neighbors, say $u$ and $w$, in $V_2$.
Again by Theorem 2.2, each of $u$ and $w$ is an endvertex of a path component in $H[V_2]$,
and such a path joins distinct components of the forest $F$. We have that $H-u-w$
is connected. Otherwise $H-v$ is disconnected, contradicting that
$H$ is 2-connected. Let $T_{uw}$ be a spanning tree of $H-u-w$, and
let $T$ be the tree obtained from $T_{uw}$ adding the vertices $u$,
$v$ and the edges $vu$ and $vw$. However, $T$ is a spanning tree of
$G$ without a perfect matching, a contradiction.
\end{proof}

\begin{lemma}
For $G\in \mathcal{G}$ and $F\in \mathcal{G}$ with $uv\in
E(G)$ and $d_G(v)=1$, let $G'$ be the graph obtained from $G-v$ and $F$
by identifying a vertex $w$ of $F$ to $u$. Then $G' \in \mathcal{G}$.
\end{lemma}

\begin{proof} We proceed by induction on the order $n$ of $G$. If $n=2$, then $G'\cong F$ and the
result is trivial. Now let $n\geq 4$. Then $G$ is neither $K_2$ nor an even cycle. By the definition of graph class $\cal G$, there
exists a connected graph $H$ of order $p\geq 2$ with vertex set
$\{u_1, \ldots, u_p\}$ and $F_i\in \mathcal{G}$ for each $i\in\{1,
\ldots, p\}$, where $H, F_1, \ldots, F_p$ are pairwise
vertex-disjoint such that $G=H[F_1, \ldots, F_p]$. Without loss of
generality, let $v\in V(F_1)$. If $F_1=K_2$, we are done, because
$G'=H[F,F_2, \ldots, F_p]$. So  assume that $F_1\not\cong K_2$. Let
$F_1'$ be the graph obtained from $F_1-v$ and $F$ by identifying vertex $u$ of $F_1-v$
and  vertex $w$ of $F$. Note that $F_1$ has less vertices than $G$. By the
induction hypothesis, $F_1'\in \mathcal{G}$, and thus $G'=H[F_1',F_2,
\ldots, F_p]\in \mathcal{G}$.
\end{proof}

\section{Proof of Theorem 1.1}

 By the definition of a Kekul\'e set we can see that
 $G$ has no anti-Kekul\'{e} set if and only if for each $S\subseteq
E(G)$ with  $G-S$ being connected, $G-S$ has a perfect
matching. Since $G-S$ is always a spanning subgraph of $G$, the latter can be expressed as ``each connected spanning subgraph of $G$ has a perfect matching".
 So  statements (1) and (2) are equivalent. Further, the equivalence of (2) and (3) is evident.

Next we mainly show  the equivalence of statements (3) and (4). We proceed  by induction on the
order $n$ of $G$.

We first consider $(4)\Rightarrow (3)$. If $G\cong
K_2$ or $G$ is an even cycle, then each spanning tree of $G$ is
isomorphic to $P_n$, and thus has a perfect matching. By the
definition of $\cal G$ we  assume that $G\cong H[F_1, \ldots, F_p]$, where
$H$ is a connected graph of order $p\geq 2$ and $F_i\in
\mathcal{G}$, $i=1,2,\ldots,p$. For any spanning tree $T$ of $G$, let $T_i=T\cup F_i$
for $i=1, 2, \ldots, p$. Then $T_i$ is a spanning tree of $F_i$.
Since $F_i\in \mathcal{G}$, by the induction hypothesis, $T_i$ has a
perfect matching $M_i$. So, $M=\cup_{i=1}^p M_i$ is a perfect
matching of $T$. This shows $(4)\Rightarrow (3)$.

To show $(3)\Rightarrow (4)$, we assume that each spanning tree of
$G$ has a perfect matching. If $G$ is nonseparable, then by Lemma
2.3, $G$ is an even cycle or $K_2$, and thus $G\in \mathcal{G}$. So in the following
we always assume that $G$ is separable. 
We consider two cases.

\vspace{3mm}
{\bf Case 1.} There exists a separation $\{G_1, G_2\}$ of $G$ with
$n_2\geq 4$.

Let $v$ be the common vertex of $G_1$ and $G_2$, and let $G'_1$ be the graph obtained from
$G_1$ by joining a new vertex $v_2$ to $v$ with an edge.

We assert that each spanning tree of $G_2$ (resp. $G'_1$) has a perfect matching.
Let $T'_1$ be any spanning tree of $G'_1$ and $T_2$ be a spanning
tree of $G_2$. Then $(T_1'-v_2)\cup T_2$ is a spanning tree of $G$
and thus has a perfect matching $M$. Since the order of $T_2$ is
even, $v$ is matched with a vertex in $V(T_2)$ under $M$. Thus $M\cap E(T_2)$ is a perfect matching of $G_2$ and
$M\cap E(T_1'-v_2)$ is a perfect matching of $T_1'-v-v_2$. The latter implies that
$(M\cap E(T_1'-v_2))\cup \{vv_2\}$ is a perfect matching of $T_1'$. So the assertion holds. Since $G_1'$ and $G_2$ each has fewer vertices than $G$,
by the induction hypothesis we have that $G'\in \mathcal{G}$ and $G_2\in \mathcal{G}$. This along with Lemma 2.4 imply that
$G\in \mathcal{G}$.

\vspace{3mm}{\bf Case 2.} $n_2=2$ for any separation $\{G_1, G_2\}$
of $G$.

In this case, we will show that $G\cong H\circ K_1$, where $H$ is a
nonseparable graph. Hence $G\in \cal G$. To this end, we have the following claim.

\vspace{3mm}\noindent{\bf Claim.}  For a cut vertex $v$ of $G$, \\
(i) $G-v$ has exactly two
components, one of which is a single vertex $u$;\\
(ii) For any $w\in N_G(v)$ other than $u$,
we have $d_G(w)\geq 2$ and $w$ is also a cut vertex of
$G$.

\begin{proof}Let $\{G_1, G_2\}$ be a separation of $G$ with $V(G_1)\cap
V(G_2)=\{v\}$. Since $n_2=2$,  $G-v$ has a single
vertex as one component. Moreover, since $G$ has a perfect matching, by Theorem \ref{tutter} we have that exactly one
component of $G-v$ is a single vertex $u$ and all other components
of $G-v$ are even. If $G-v$ has at least three components, we take
an even component $G'$ of $G-v$. Then $\{G[\{v\}\cup V(G')], G-V(G')\}$
is a separation of $G$ such that $G-V(G')$ has  an even order at least 4, contradicting the assumption of Case 2. This shows (1).

Now we show (2). If $w$ is not a cut vertex of $G$, then $G$ has a spanning
tree $T$ in which both $u$ and $w$ are leaves adjacent to $v$. It is
clear that $c_o(T-v)\geq 2$. By Theorem \ref{tutter}, $T$ has no perfect
matching, a contradiction. This proves the Claim.
\end{proof}

Let $H$ be the graph obtained
from $G$ by deleting all  vertices of degree one. Hence $H$ is connected. Note that $G$ has cut vertices,
and by Claim (i) each cut vertex $v$ of $G$ has degree at least two and is adjacent to a vertex $u$ of degree one. Claim (ii) implies that each vertex of $N_G(v)\setminus\{u\}$ is a cut vertex of degree at least two in $G$. Hence $N_G(v)\setminus\{u\}\subseteq V(H)$, and $H$ has at least two vertices. Since $H$ is connected, it follows that each vertex  of $G$ with
degree at least two is a cut vertex.
That is,  each vertex $v$ of $H$ is a cut vertex of $G$.  By  Claim (i), we have that $H-v$ is connected and $v$ is adjacent to one vertex of degree one in $G$.  The former shows that $H$ is nonseparable, and the latter implies that $G=H[F_1, \ldots, F_p]\in \mathcal{G}$,
where $p\geq 2$ is the order of $H$ and $F_i\cong K_2$ for each $i$. In other words, $G\cong H\circ K_1$. The proof is completed. \hfill$\Box$

\section{Concluding remarks}

In this note, we characterize the graphs without an anti-kekule set, which are exactly those graphs
whose each spanning tree has a perfect matching.

Note that every tree with a perfect matching belong to
$\mathcal{G}$. It is well known (see, page 80 \cite{Bondy}) that a
tree has a perfect matching if and only if $c_o(T-v)=1$ for all
vertex $v\in V(T)$. Indeed, it is not hard to show that the family
$\mathcal{T}$ of trees with a perfect matching can be recursively
constructed by the following way:

(1) $K_2\in \mathcal{T}$,

(2) if $T_i\in \mathcal{T}$ for $1\leq i\leq 2$ with $V(T_1)\cap
V(T_2)=\emptyset$, then $T\in \mathcal{T}$, where $T$ is obtained
from $T_1$ and $T_2$ by joining a vertex of $T_1$ to that of $T_2$.

\vspace{3mm} It is natural to ask that what is largest size of a
graph of order $2n$ whose each spanning tree has a perfect matching? The answer is clear for $n\leq 3$.
It is easy to verify that $K_2$ and $C_4$ has the largest size for $n=1$ and $n=2$, respectively.
If $n=3$, there are exactly two graphs, i.e., $C_6$ and $K_3\circ K_1$, with the desired property.
In general, we have the following.

\begin{corollary}
If $G$ is a connected graph of order $2n$ for a positive integer
$n$ and size $m$ whose each spanning tree has a perfect
matching, then $m\leq f(n)$, where
$$
f(n)= \left \{
\begin{array}{ll}
1, & \mbox{if $n=1$, \ }\\
4, & \mbox{if $n=2$, \ }\\
\frac {(n+1)n} 2, &\mbox{if $n\geq 3$, \ }
\end{array}
\right.
$$
with equality if and only if
$$
G\cong \left \{
\begin{array}{ll}
K_2, & \mbox{if $n=1$, \ }\\
C_4, & \mbox{if $n=2$, \ }\\
C_6\ or\ K_3\circ K_1, &\mbox{if $n=3$, \ }\\
K_n\circ K_1, &\mbox{if $n\geq 4$. \ }\\
\end{array}
\right.
$$

\end{corollary}
\begin{proof} Assume that $G$ is a connected graph of order $2n$ and size $m$ whose each spanning tree has a perfect
matching. By the observation before this corollary, the result holds for $n\leq 3$. Next, we further assume
that $n\geq 4$ and $m$ is as large as possible. We shall show that $G\cong K_n\circ K_1$.

If $G$ is nonseparable, then by Lemma 2.3,
$$
m= \left \{
\begin{array}{ll}
1, & \mbox{if $n=1$, \ }\\
2n, &\mbox{if $n\geq 2$. \ }
\end{array}
\right.
$$
But, $m<\frac {(n+1)n} 2$ for any $n\geq 4$, contradicting our choice. So, $G$ must be separable.
By Theorem 1.1, there
exists a connected graph $H$ of order $p\geq 2$ with vertex set
$\{u_1, \ldots, u_p\}$ and $F_i\in \mathcal{G}$ for each $i\in\{1,
\ldots, p\}$, where $H, F_1, \ldots, F_p$ are pairwise
vertex-disjoint such that $G=H[F_1, \ldots, F_p]$. By the maximality of $G$, $H$ must be complete.
It remains to show that $F_i\cong K_2$ for all $i$. Toward a contradiction, suppose that $n_p\geq 4$, without loss of generality.


Let $H'=(H-u_p) \vee K_a$, where $a=\frac {n_p} 2$ and $G'=H'[F_1', \ldots, F'_{p+a-1}]$, where
$$
F_i'= \left \{
\begin{array}{ll}
F_i, & \mbox{if $1\leq i\leq p-1$, \ }\\
K_2, &\mbox{if $p\leq i\leq p+a-1$. \ }
\end{array}
\right.
$$
Then
\begin{eqnarray*}
m'&=&e(H')+\sum_{i=1}^{p+a-1}e(F_i')\\
  &=& e(H)-d_H(u_p)+a(p-1)+\frac {(a+1)a} 2+\sum_{i=1}^{p-1}e(F_i)\\
  &=& e(H)+(a-1)(p-1)+\frac {(a+1)a} 2+\sum_{i=1}^{p-1}e(F_i)\\
  &=& m+(a-1)(p-1)+\frac {(a+1)a} 2-e(F_p).\\
\end{eqnarray*}
Note that $(a-1)(p-1)+\frac {(a+1)a} 2-e(F_p)\geq 0$, with equality if and only if $p=2$ and $F_p\cong C_4$. 
Again by the maximality of $G$, $m'=m$. Thus $p=2$, $F_p\cong C_4$, and $F_i=K_2$ for each $i\leq p-1$, otherwise, by repeating the argument above, we obtain a graph $G''\in \mathcal{G}$ with size greater than that of $G$. It follows that $n=3$, contradicting the assumption that $n\geq 4$. So, $F_i\cong K_2$ for all $i$, together with $H\cong K_n$, we conclude that $G\cong K_n\circ K_1$.

The proof is completed.  

\end{proof}

\vspace{4mm} \noindent{\bf Acknowledgements}

\vspace{2mm} The authors are grateful to Xiaodong Liang, Wei Xiong,
Xinhui An for their careful helpful comments.

\end{document}